\documentclass[12pt]{amsart}
\usepackage{amsmath,amsxtra,amstext,amssymb,latexsym,amsthm,dsfont,amscd,xypic,cite}

\theoremstyle{plain}
\newtheorem{thm}{Theorem}[section]
\newtheorem{lemma}[thm]{Lemma}
\newtheorem{cor}[thm]{Corollary}
\newtheorem{prop}[thm]{Proposition}
\newtheorem{definition}[thm]{Definition}
\newtheorem{rem}[thm]{Remark}

\def\lr{{\mathcal{LR}}}
\def\Chi{X}
\def\tG{{\tilde G}}\def\tnu{{\tilde \nu}}\def\tH{{\tilde H}}\def\tP{{\tilde P}}
\def\tS{{\tilde S}}\def\tw{{\tilde w}}\def\tW{{\tilde W}}
\def\trho{{\tilde \rho}}
\def\Wt{{\rm Wt}}
\def\tLambda{\tilde\Lambda}

\newcommand{\bull}{_{\bullet}}
\newcommand{\SL}{{\rm{SL}}}
\newcommand{\SO}{{\rm{SO}}}

\newcommand{\gothG}{\mathfrak{g}}
\newcommand{\gothB}{\mathfrak{b}}
\newcommand{\gothH}{\mathfrak{h}}
\newcommand{\gothP}{\mathfrak{p}}
\newcommand{\gothL}{\mathfrak{l}}
\newcommand{\tgothG}{\tilde{\mathfrak{g}}}
\newcommand{\tgothB}{\tilde{\mathfrak{b}}}
\newcommand{\tgothH}{\tilde{\mathfrak{h}}}
\newcommand{\tgothP}{\tilde{\mathfrak{p}}}
\newcommand{\tgothL}{\tilde{\mathfrak{l}}}
\newcommand{\codim}{{\rm codim}}

\newcommand{\image}{{\rm Im}}
\newcommand{\diag}{{\rm diag}}

\newcommand{\Gr}{{\rm Gr}}

\newcommand{\OG}{{\rm OG}}

\newcommand{\Fl}{{\rm F\ell}}

\newcommand{\tilG}{\tilde{G}}
\newcommand{\tilH}{\tilde{H}}
\newcommand{\tilB}{\tilde{B}}
\newcommand{\tilP}{\tilde{P}}
\newcommand{\tilR}{\tilde{R}}
\newcommand{\tilT}{\tilde{T}}

\newcommand{\tilW}{\tilde{W}}

\newcommand{\Plam}{P(\lambda)}
\newcommand{\tilPl}{\tilde P(\lambda)}
\newcommand{\Llam}{L(\lambda)}
\newcommand{\tilLl}{\tilde L(\lambda)}
\newcommand{\tcalT}{\tilde{\mathcal{T}}}
\newcommand{\calT}{\mathcal{T}}

\newcommand{\ZZ}{{\mathbb Z}}\newcommand{\QQ}{{\mathbb Q}}
\newcommand{\Dim}{{\rm Dim}}
\newcommand{\CoDim}{{\rm CoDim}}

\def\Ker{{\rm Ker}}

\begin{document}

\title[Generalizations of $(H^*(G/P),\odot_0)$]
{Branching Schubert calculus and the Belkale-Kumar product on cohomology}

\author{Nicolas Ressayre}
\author{Edward Richmond}

\address{Universit´e Montpellier II\\ D´epartement de Math´ematiques\\ Case courrier 051-Place Eug`ene Bataillon\\ 34095 Montpellier Cedex 5\\France}
\email{resayre@math.univ-montp2.fr}

\address{Department of Mathematics\\ University of Oregon\\ Eugene,  OR 97402}
\email{erichmo2@uoregon.edu}

\maketitle

\begin{abstract} In \cite{BK06}, Belkale and Kumar define a new product on the cohomology of flag varieties and use this new product to give an improved solution to the eigencone problem for complex reductive groups.  In this paper, we give a generalization of the Belkale-Kumar product to the branching Schubert calculus setting.  The study of Branching Schubert calculus attempts to understand the induced map on cohomology of an equivariant embedding of flag varieties.  The main application of our work is a compact formulation of the solution to the branching eigencone problem.\end{abstract}

\section{Introduction}

\quad Let $G$ be a connected complex reductive group and let $\tilG$ be a connected reductive subgroup of $G$.  Let $i:\tilG\hookrightarrow G$ denote the embedding of groups.
 For any one parameter subgroup $\lambda:\mathbb{C}^*\rightarrow \tilG$, we have the corresponding parabolic subgroup

$$\tilPl:=\{g\in \tilG\ |\ \lim_{t\rightarrow 0} \lambda(t)g\lambda(t)^{-1}\ \mbox{exists in}\  \tilG\}.$$

Similarly, we define $\Plam:=P(i\circ\lambda)\subseteq G$.  Let $W_P\subseteq W$ denote the Weyl groups of $\Plam$ and $G$respectively.  For any $w\in W^P\simeq W/W_P$, let $\Lambda_w\subseteq G/\Plam$ denote the corresponding Schubert variety and let $[\Lambda_w]\in H^*(G/\Plam)=H^*(G/\Plam,\mathbb{Z})$ denote the Schubert class of $\Lambda_w$.  We also have Schubert varieties $\Lambda_{\tilde w}\subseteq\tilG/\tilPl$ and Schubert classes $[\Lambda_{\tilde w}]\in H^*(\tilG/\tilPl)$ for any $\tilde w\in \tilW^P\simeq \tilW/\tilW_P$.  Consider the $\tilG$-equivariant map of flag varieties $$\phi_{\lambda}:\tilG/\tilPl\hookrightarrow G/\Plam.$$

The problem concerning ``branching Schubert calculus" is to compute the comorphism $$\phi_{\lambda}^*([\Lambda_w])=\sum_{\tilde w\in\tilW^P}d_w^{\tilde w}[\Lambda_{\tilde w}]$$ in terms of the basis of Schubert classes in $H^*(\tilG/\tilPl)$.  Observe that if $G/\Plam=\tilG/\tilPl\times\tilG/\tilPl$ and $\phi_{\lambda}$ is the diagonal embedding, then

$$\phi_{\lambda}^*([\Lambda_{\tilde u}\times\Lambda_{\tilde v}])=[\Lambda_{\tilde u}]\cdot[\Lambda_{\tilde v}].$$

\quad In \cite{BK06}, Belkale and Kumar define the ring $(H^*(G/\Plam),\odot_0)$.
Additively, this ring is the same as $H^*(G/\Plam)$.
In section \ref{section main}, we construct a map
$$
\phi_{\lambda}^{\odot}:H^*(G/\Plam)\rightarrow H^*(\tilG/\tilPl)
$$ from $\phi_{\lambda}^*$.
This map is a generalization of Belkale-Kumar product in the sense that if we consider
the diagonal embedding where $G/\Plam=\tilG/\tilPl\times\tilG/\tilPl$ we have that
$$
\phi_{\lambda}^{\odot}([\Lambda_{\tilde u}\times\Lambda_{\tilde v}])=
[\Lambda_{\tilde u}]\odot_0[\Lambda_{\tilde v}].
$$
In general, cohomology equipped with $\odot_0$ is not functorial.
Our main result is on the functoriality of $\phi_{\lambda}^{\odot}$ with respect
to the Belkale-Kumar product $\odot_0$ and its relationship with the natural map
$\phi^*_{\lambda}$ on cohomology.
For any $(w,\tilde w)\in W^P\times \tilW^P$, define the structure constants
$c_w^{\tilde w}, d_w^{\tilde w}\in\mathbb{Z}_{\geq 0}$ by the comorphisms
$$
\phi_{\lambda}^{\odot}([\Lambda_w])=
\sum_{\tilde w\in\tilW^P}c_w^{\tilde w}[\Lambda_{\tilde w}]
$$
and
$$
\phi_{\lambda}^*([\Lambda_w])=\sum_{\tilde w\in\tilW^P}d_w^{\tilde w}[\Lambda_{\tilde w}].
$$

\begin{thm}\label{Thm1}
The map $\phi_{\lambda}^{\odot}$ is a graded ring homomorphism from $(H^*(G/\Plam),\odot_0)$
on $(H^*(\tilG/\tilPl),\odot_0)$.
\bigskip
Moreover, if $c_w^{\tilde w}\neq 0$, then $c_w^{\tilde w}=d_w^{\tilde w}$.\end{thm}

The proof of the above theorem requires a modification on the construction of the Belkale-Kumar product in \cite{BK06}.
In \cite{Re07}, the first author gives a minimal list of inequalities which characterize the eigencone of the pair $\tilG\subseteq G$.  In section \ref{section eigen}, we use the comorphism $\phi_{\lambda}^{\odot}$ to give a more elegant formulation of this statement.

\section{Preliminaries and Levi-movability}

\quad Fix maximal tori $\tilH\subseteq H$ of $\tilG$ and $G$ respectively such that $\image(\lambda)\subseteq\tilH$.
Furthermore, fix Borel subgroups $\tilB$ and $B$ of $\tilG$ and $G$ respectively such that $\tilH\subseteq\tilB\subseteq\tilPl$ and $H\subseteq B\subseteq \Plam$ and $\tilB=\tilG\cap B$.
Observe that such Borel subgroups always exist by choosing an appropriate generic rational
one parameter subgroup $\lambda'$ close to $\lambda$ and setting
$\tilB=\tilP(\lambda')$ (resp. $B=P(\lambda')$).
Let $W_P\subseteq W$ denote the Weyl groups of $\Plam$ and $G$ respectively and let $W^P$ denote the set of minimal length representatives of $W/W_P$.
For any $w\in W^P$, we define the shifted Schubert variety $$\Lambda_w:=\overline{w^{-1}Bw\Plam/\Plam}.$$  The cohomology classes $\{[\Lambda_w]\}_{w\in W^P}$ form an additive basis for $H^*(G/\Plam).$  For any $\tilde w\in \tilW^P\simeq \tilW/\tilW_P$, we will denote the corresponding Schubert variety in $\tilG/\tilPl$ by $\Lambda_{\tilde w}$.

\subsection{A generalization of Levi-movability}

\quad Our discussion begins with a generalized notion of Levi-movable defined in \cite{BK06}.
Define the Levi subgroup $\Llam\subseteq\Plam$ to be the centralizer of $\image(\lambda)$ in $G$.  For any $w\in W^P$, consider the comorphism $$\phi_{\lambda}^*([\Lambda_w])=\sum_{\tilde w\in\tilW^P}d_w^{\tilde w}[\Lambda_{\tilde w}],$$ expanded in the Schubert basis.  Let $w_0, w_P$ denote the longest elements in $W$ and $W_P$ respectively (we also have longest elements $\tilde w_0$ and $\tilde w_P$ in $\tilW$ and $\tilW_P$ accordingly) and for any $w\in W^P$ (resp. $\tilde w\in \tilW^P$), let $w^{\vee}:=w_0ww_P\in W^P$ (resp. $\tilde w^{\vee}\in \tilW^P$).  By Kleiman's tranversality \cite{Kl74}, if the coefficient $d_w^{\tilde w}\neq 0$, then it can be realized as the cardinality of the intersection of translates $$|\phi_{\lambda}^{-1}(g\Lambda_w)\cap \tilde g\Lambda_{\tilde w^{\vee}}|=d_w^{\tilde w}$$ in $\tilG/\tilPl$ for generic $(g,\tilde g)\in G\times\tilG$.  The following lemma is proved in \cite{BK06}:

\begin{lemma}\label{Lemma gtop}If $eP\in g\Lambda_w$, then there exists a $p\in \Plam$, such that $g\Lambda_w=p\Lambda_w$.\end{lemma}

Let $T$ and $\tilT$ denote the tangent spaces of $G/\Plam$ and $\tilG/\tilPl$ at the identity and for any $(p,\tilde p)\in\Plam\times\tilPl$ and $(w,\tilde w)\in W^P\times\tilW^P$  let $pT_w$ and $\tilde p\tilT_{\tilde w}$ denote the tangent spaces of $p\Lambda_w$ and $\tilde p\tilde\Lambda_{\tilde w}$ at the identity.
 Assume that
$$\codim(\Lambda_w;G/\Plam)=\codim(\tilde\Lambda_{\tilde w};\tilG/\tilPl).$$
Otherwise, $d_w^{\tilde w}=0$.
By Lemma \ref{Lemma gtop}, the coefficient $d_w^{\tilde w}\neq 0$ if and only if the intersection $$\phi_{\lambda}^{-1}(p\Lambda_w)\cap \tilde p\Lambda_{\tilde w^{\vee}}$$ is transverse at the point $e\tilPl\in\tilG/\tilPl$ for generic $(p,\tilde p)\in\Plam\times\tilPl.$  This is equivalent to having an isomorphism on the map between tangent spaces

\begin{equation}\label{eqn Tangent}\tilT \rightarrow \frac{T}{pT_w}\oplus\frac{\tilT}{\tilde p\tilT_{\tilde w^{\vee}}}\end{equation} given by $v\mapsto (\overline{(\phi_{\lambda})_*(v)},\overline v)$, for generic $(p,\tilde p)\in\Plam\times\tilPl$.
 The following definition is a generalization of Levi-movable and is given in \cite{Ri09}.

\begin{definition}
We say $(w,\tilde w)\in W^P$ is Levi-movable with respect to $\phi_{\lambda}$
if for generic $(l,\tilde l)\in\Llam\times\tilLl$ the following natural map on
tangent spaces is an isomorphism:
$$\tilT \rightarrow \frac{T}{lT_w}\oplus\frac{\tilT}{\tilde l\tilT_{\tilde w}^\vee}.$$\end{definition}

Observe that if $(w,\tilde w^{\vee})\in W^P$ is Levi-movable with respect to $\phi_{\lambda}$, then  $d_w^{\tilde w}\neq 0$.  The converse is not true in general.

\subsection{The Belkale-Kumar numerical criterion}

\quad We now want to explain how the Belkale-Kumar numerical criterion can be generalized
to our setting.
We first establish some notation for root systems associated to Lie algebras.
Denote the Lie algebras of groups $G,H,B,\Plam,\Llam$ by the corresponding frankfurt letters $\gothG,\gothH,\gothB,\gothP,\gothL_P$.
 Similarly we have Lie algebras $\tgothG,\tgothH,\tgothB,\tgothP,\tgothL_P$ for subgroups of $\tilG$.

\smallskip

\quad Let $R\subseteq \gothH^*$ be the set of roots and let $R^{\pm}\subseteq R$ denote the set of positive roots (negative roots) with respect to the Borel subgroup $B$.
 Let $\Delta=\{\alpha_1,\ldots,\alpha_n\}$ denote the simple roots in $R$.
 Let $R_P$ denote the set of roots corresponding to $\gothL_P$ and let $R_P^{\pm}$ denote the set of positive roots (negative roots) with respect to the Borel subgroup $B_P:=B\cap \Llam$ of $\Llam$.
 Let $\Delta(P)$ be the set of simple roots that generate $R_P^+$.
 Similarly, we have the roots $\tilR, \tilR^{\pm}, \tilR_P,\tilR_P^{\pm},\tilde\Delta, \tilde\Delta(P)\subseteq\tgothH^*.$

\smallskip

\quad The following character is defined in \cite{BK06} and will play an important role in constructing $\phi_{\lambda}^{\odot}$.
 For $w\in W^P$, define $\chi_w\in\gothH^*$ by
$$\chi_w:=\sum_{\beta\in (R^+\backslash R_P^+)\, \cap\, w^{-1}R^+}\beta.$$
Similarly, for any $\tilde w\in\tilW^P$, we can define $\tilde\chi_{\tilde w}\in\tgothH^*$.
Define
$$\dot{\lambda}:=\frac{d}{dt}\lambda(1)\in\tgothH.$$
Observe that $\alpha(\dot\lambda)\in\mathbb{Z}$ for any $\alpha\in\tilR$ since $\lambda$ is a one parameter subgroup of $\tilH$.  Moreover, for any $\tilR^+$, we have that $\alpha(\dot\lambda)\geq 0$ with equality only when $\alpha\in\tilR^+_P$.  This implies that $\tilde{\chi}_{\tilde w}(\dot\lambda)$ is integral and non negative.  Likewise, we have that $i^*(\chi_w)(\dot\lambda)$ is also integral and non negative since $i\circ\lambda$ is a one parameter subgroup of $H$.  Here we are abusing notation by letting $i:\tgothH\hookrightarrow \gothH$ denote the induced map on Cartan subalgebras.  These characters are connected to the tangent spaces given in \eqref{eqn Tangent} in the sense that $\gothH$ acts on complex line $\displaystyle\det\left(T/T_w\right)$ by multiplication by $\chi_{w}$.

\begin{prop}\label{Prop_char}Let $(w,\tilde w^{\vee})\in W^P\times \tilde W^P$ such that $d_w^{\tilde w}\neq 0$.
 Then $$(i^*(\chi_w)-\tilde\chi_{\tilde w})(\dot\lambda)\leq 0.$$  Moreover, $(w,\tilde w^{\vee})$ is Levi-movable with respect to $\phi_{\lambda}$ if and only if $(i^*(\chi_w)-\tilde\chi_{\tilde w})(\dot\lambda)=0.$\end{prop}

\begin{proof}The proposition is proved in the second author's thesis \cite{RiThesis} and for the diagonal embedding by Belkale and Kumar in \cite[Theorem 15, Theorem 29]{BK06}.
 We give a sketch of the proof here.

\smallskip

\quad For any $w\in W^P$, let $\calT_w:=\Plam\times_{B_L}T_w$ denote the corresponding $\Plam$-equivariant vector bundle on $\Plam/B_L$.
 Observe that a $T_w$ is a $B_L$-module since the action of $B_L$ on $\Lambda_w$ fixes the identity.
 If $T_w=T$, then we denote $\calT_w$ by simply $\calT$.
 For any $\tilde w\in \tilW^P$, we can define analogous $\tilPl$-equivariant vector bundles $\tcalT_{\tilde w}$ on $\tilPl/\tilB_L$.
 The map on tangent spaces given \eqref{eqn Tangent} induces a $\tilPl$-equivariant map on vector bundles $$\Theta: \tcalT' \oplus \tcalT\rightarrow \calT/\calT_w\oplus\tcalT/\tcalT_{\tilde w^{\vee}}$$ on $\Plam/B_L\times \tilPl/\tilB_L$ where $\tilPl$ acts diagonally on $\tcalT':=\Plam/B_L\times\tilT$.

\smallskip

\quad If $d_w^{\tilde w}\neq 0$, then the map \eqref{eqn Tangent} is an isomorphism for generic $(p,\tilde p)\in\Plam\times\tilPl$.
 Hence the induced determinant map $\det(\Theta)$ on top exterior powers is nonzero.
 The map $\det(\Theta)$ can be viewed as a nonzero $\tilPl$-invariant section of the line bundle
$$\mathcal{L}:=(\det\tcalT' \boxtimes \det\tcalT)^*\otimes(\det\calT/\calT_w\boxtimes\det\tcalT/\tcalT_{\tilde w^{\vee}})$$ on $\Plam/B_L\times \tilPl/\tilB_L$.
 Hence the points in $\Plam/B_L\times \tilPl/\tilB_L$ are generically semi-stable with respect to action of $\tilPl$ on $\mathcal{L}$.
 The Hilbert-Mumford criterion for semi-stability implies that $$(i^*(\chi_w)+\tilde\chi_{\tilde w^{\vee}}-\tilde\chi_1)(\dot\lambda)=(i^*(\chi_w)-\tilde\chi_{\tilde w})(\dot\lambda)\leq 0.$$

\quad If $(w,\tilde w^{\vee})$ is Levi-movable with respect to $\phi_{\lambda}$, then the restriction of $\det(\Theta)$ to $\Llam/B_L\times \tilLl/\tilB_L$ is also nonzero.
 Since $\lambda$ is central acting diagonally on $\Llam\times \tilLl$, we have that $\lambda$ acts trivially on $\mathcal{L}$ restricted to $\Llam/B_L\times \tilLl/\tilB_L$.
 Hence \begin{equation}\label{Levi_cond}(i^*(\chi_w)-\tilde\chi_{\tilde w})(\dot\lambda)= 0.\end{equation}  Conversely, if \eqref{Levi_cond} is satisfied and  $d_w^{\tilde w}\neq 0$, then $\det(\Theta)$ restricted to $\Llam/B_L\times \tilLl/\tilB_L$ is nonzero.
 This implies the map \eqref{eqn Tangent} is an isomorphism for generic $(l,\tilde l)\in\Llam\times\tilLl$ and hence $(w,\tilde w^{\vee})$ is Levi-movable.\end{proof}

\subsection{Revisiting the numerical criterion}\label{sec:defTi}

\quad For the ordinary comorphism $\phi_\lambda^*$, there is an obvious numerical condition for a structure
coefficient to be non zero: namely, the dimension (or degree) condition.  We explain how Levi-movability can be checked by a multidimension condition.

\smallskip

\quad For any $i\in\ZZ$, we set $T^i:=\{\xi\in T\,:\,\lambda(t)\xi=t^i\xi\}$ and
$T_w^i=T^i\cap T_w$.
Note that $T^i=\{0\}$ for $i\geq 0$ and for almost all $i<0$.
Since the translated Schubert cells are stable by the action of $\lambda$, we have:
$$
T=\oplus_{i\in\ZZ_{<0}}T^i{\rm\ \ and\ \ }T_w=\oplus_{i\in\ZZ_{<0}}T^i_w.
$$
In the same way we define $\tilde T^i$ and $\tilde T_{\tilde w}^i$.
Now, for all $i\in\ZZ_{<0}$, we set
$d^i=\dim T^i$, $\delta^i_w=d^i-\dim T_w^i$,
$\tilde d^i=\dim \tilde T^i$ and
$\delta_{\tilde w}^i=\tilde d^i-\dim \tilde T_{\tilde w}^i$.
We now form the following vector dimension and codimension:
$$
\Dim(\tilde T)=\bigg (d^i\bigg)_{i\in\ZZ_{<0}},\
\CoDim(\tilde T_{\tilde w})=\bigg (\delta_{\tilde w}^i\bigg)_{i\in\ZZ_{<0}}
\ {\rm and\ \ }
\CoDim(T_w)=\bigg (\delta_w^i\bigg)_{i\in\ZZ_{<0}}.
$$

\begin{prop}
\label{Prop_Dim}
Let $(w,\tilde w^{\vee})\in W^P\times \tilde W^P$ such that $d_w^{\tilde w}\neq 0$.
Then the following are equivalent
\begin{enumerate}
\item $(w,\tilde w^{\vee})$ is Levi-movable with respect to $\phi_{\lambda}$;
\item $\Dim(\tilde T)=\CoDim(\tilde T_{\tilde w})+\CoDim(T_w)$.
\end{enumerate}
\end{prop}

\begin{proof}
  Let us first assume that $(w,\tilde w^{\vee})$ is Levi-movable with respect to $\phi_{\lambda}$.
Let $(l,\tilde l)\in\Llam\times\tilLl$ be such that the natural map
$$\tilT \rightarrow \frac{T}{lT_w}\oplus\frac{\tilT}{\tilde l\tilT_{\tilde w}}$$
is an isomorphism.
Since this linear map is $\lambda$-equivariant, it induces an isomorphism between
each $\lambda$-eigenspace.
Then, the equality of the vector dimensions in the proposition follows
from the fact that $\lambda$ commutes with $l$ and $\tilde l$.

\smallskip

\quad Conversely, let us assume that $\Dim(\tilde T)=\CoDim(\tilde T_{\tilde w})+\CoDim(T_w)$.
One easily checks that
$(i^*(\chi_w)-\tilde\chi_{\tilde w})(\dot\lambda)=
\sum_id^i-\sum_i(\delta_{\tilde w^i}+\delta_w^i)=0$.
Now, the result follows from  Proposition~\ref{Prop_char}.
\end{proof}

\begin{rem}
  Proposition~\ref{Prop_Dim} can be applied with any one parameter subgroup giving
$\tilde P$ and $P$.
To obtain  optimal decompositions of $\tilde T$ and $T$  one should choose
a generic one parameter subgroup giving $\tilde P$ and $P$.
\end{rem}

\subsection{The Azad-Barry-Seitz theorem}

\quad In this subsection, we explain how the  Azad-Barry-Seitz theorem (see \cite{AzBaSe})
gives another interpretation of the $T^i$'s in the case of $G\subset G\times G$ (we omit the tilde above $G$ for simplicity).

\smallskip

\quad We are interested in the action of $L(\lambda)$ on $T=\gothG/\gothP$.
For any $\alpha\in R$, we denote by $\gothG_\alpha$ the eigenspace in
$\gothG$ of weight $\alpha$ for $H$.
Since $T$ has no multiplicity for the action of $H$, it has no multiplicity for the action
of $L(\lambda)$ and, hence, has  a canonical decomposition  $T=\oplus_jV_j$ as a sum of irreducible $L(\lambda)$-modules.  Since $H\subset L(\lambda)$, each $V_i$ is a sum of $\gothG_\alpha$ for some $\alpha\in R^+\backslash R^+_P$:
the decomposition $T=\oplus_jV_j$ corresponds to a partition $R^+\backslash R^+_P=\bigsqcup_jR_j$.

Let $\beta$ and $\beta'$ be two negative roots.
We write
\begin{eqnarray}
  \label{eq:beta}
\beta=\sum_{\alpha\in \Delta(P)}c_\alpha\alpha + \sum_{\alpha\in\Delta\backslash\Delta(P)}d_\alpha\alpha,
\end{eqnarray}
with $c_\alpha,d_\alpha\in\ZZ_{\leq 0}$.
We also write $\beta'$ in the same way with some $c_\alpha'$ and $d_\alpha'$.
We write $\beta\equiv\beta'$ if and only if $\sum_{\alpha\in\Delta\backslash\Delta(P)}d_\alpha\alpha=
\sum_{\alpha\in\Delta\backslash\Delta(P)}d_\alpha'\alpha$.
The relation $\equiv$ is obviously an equivalence relation.
Let $S$ denote the set of equivalence classes in  $R^+\backslash R^+_P$ for $\equiv$.
We can now rephrase the main result of \cite{AzBaSe}:

\begin{thm}[Azad-Barry-Seitz]
\label{th:abs}
  For any $s\in S$, $V_s:=\oplus_{\alpha\in s}\gothG_\alpha$ is an irreducible $L(\lambda)$-module.
In particular, $\bigsqcup_j R_j$ is the partition in equivalence classes for $\equiv$.
\end{thm}

An interesting consequence is the following corollary.

\begin{cor}
  \label{cor:abs}
For $\lambda$ generic such that $P=P(\lambda)$, the subspaces $T^i$ defined in Section~\ref{sec:defTi}
are irreducible $L(\lambda)$-modules.
\end{cor}

\begin{proof}
\quad Consider  the center $Z$ of $L(\lambda)$ and its neutral component $Z^\circ$.
By the theorem, it is sufficient to prove that $\beta\equiv\beta'$ if and only if
$\langle\lambda,\beta\rangle=\langle\lambda,\beta'\rangle$.
There exists an open subset of $\lambda$ in $Y(Z^\circ)\otimes\QQ$  such that $P=P(\lambda)$.
So, for $\lambda$ generic, we have for all pairs $(\beta,\beta')\in R^2$,
$$\langle\lambda,\beta\rangle=\langle\lambda,\beta'\rangle\ \mbox{if and only if}\ \beta_{|Z^\circ}=\beta'_{|Z^\circ}.$$

\quad Under the action of $Z^\circ$, we have a decomposition
$$\gothG/\gothP=\oplus_{\chi\in X(Z^\circ)}V_\chi,$$ as sum of eigenspaces.
Since $Z^\circ$ is central in $L(\lambda)$, each $V_\chi$ is $L(\lambda)$-stable.  Note that $Z^\circ\subset Z\subset H$; and more precisely $$Z=\bigcap_{\alpha\in\Delta(P)}\Ker\alpha.$$
It follows that  the family $(\alpha_{|Z^\circ})_{\alpha\in\Delta\backslash\Delta(P)}$ is free.
For $\beta$ as in Equation~\ref{eq:beta}, we have
$\beta_{|Z^\circ}=\sum_{\alpha\in\Delta\backslash\Delta(P)}d_\alpha\alpha_{|Z^\circ}$.
We obtain that
$$\beta\equiv\beta'\iff\beta_{|Z^\circ}=\beta'_{|Z^\circ}.$$\end{proof}

\section{The main result}\label{section main}

\quad In this section we define the map $\phi_{\lambda}^{\odot}$ on cohomology and prove Theorem \ref{Thm1}.  This construction is analogous to the construction of the Belkale-Kumar product in \cite[Section 6]{BK06}.  For any $(u,v,w)\in (W^P)^3$ define the usual structure coefficients $d_{u,v}^w$ by the usual cohomology product

$$[\Lambda_u]\cdot[\Lambda_v]=\sum_{w\in W^P} d_{u,v}^w[\Lambda_w].$$
Similarly, we have structure coefficients $d_{\tilde u,\tilde v}^{\tilde w}$ for $H^*(\tilG/\tilPl)$.
 Let the symbol $\tau$ denote an indeterminant and consider the $\mathbb{Z}$-module
$H^*(G/\Plam)\otimes_{\mathbb{Z}} \mathbb{Z}[\tau].$  We define a product structure by  $$[\Lambda_u]\odot\bull[\Lambda_v]:=\sum_{w\in W^P} (\tau^{(\chi_w-\chi_u-\chi_v)(i(\dot\lambda))})\ d_{u,v}^w[\Lambda_w].$$  We extend this product $\mathbb{Z}[\tau]$-linearly to all of $H^*(G/\Plam)\otimes_{\mathbb{Z}} \mathbb{Z}[\tau]$.  We also define $\odot\bull$ on $H^*(\tilG/\tilPl)\otimes_{\mathbb{Z}} \mathbb{Z}[\tau]$ using the characters $\tilde\chi_{\tilde w}$ and replacing $i(\dot\lambda)$ by $\dot\lambda$.  By \cite[Proposition 17]{BK06}, this product structure is well defined, commutative and associative.
 We remark that the product $\odot\bull$ is very similar to the product $\odot$ defined in \cite{BK06} by Belkale and Kumar.
 The main difference is that $\odot\bull$ uses the single indeterminant $\tau$ where as $\odot$ uses several indeterminants, one for each simple root in $\Delta\backslash\Delta(P)$.

\begin{lemma}\label{lemma_bkprod}The product $[\Lambda_u]\odot\bull[\Lambda_v]\big|_{\tau=0}=[\Lambda_u]\odot_0[\Lambda_v]$ where $\odot_0$ denotes the Belkale-Kumar product.\end{lemma}

\begin{proof}By the definition of $\odot_0$ found in \cite[Section 6]{BK06}, it suffices to show that $\alpha(i(\dot\lambda))> 0$ for all $\alpha\in\Delta\backslash\Delta(P)$ and $\alpha(i(\dot\lambda))=0$ for all $\alpha\in\Delta(P)$.
 This is immediate from the definition of $P=\Plam$.\end{proof}

Recall that for any $(w,\tilde w)\in W^P\times \tilde W^P$ the structure coefficients $d_w^{\tilde w}$ of the map $\phi_{\lambda}^*$ are defined by expanding in the Schubert basis

$$\phi_{\lambda}^*([\Lambda_w])=\sum_{\tilde w\in\tilW^P}d_w^{\tilde w}[\Lambda_{\tilde w}].$$  We define the $\mathbb{Z}[\tau]$-linear map

$$\phi_{\lambda}^{\bullet}:H^*(G/\Plam)\otimes_{\mathbb{Z}} \mathbb{Z}[\tau]\rightarrow H^*(\tilG/\tilPl)\otimes_{\mathbb{Z}} \mathbb{Z}[\tau]$$ by

$$\phi_{\lambda}^{\bullet}([\Lambda_w]):=\sum_{\tilde w\in\tilW^P}(\tau^{\tilde\chi_{\tilde w}(\dot\lambda)-\chi_w(i(\dot\lambda))})\ d_w^{\tilde w}[\Lambda_{\tilde w}].$$

For the rest of this section, we will denote $\chi_w(i(\dot\lambda))$ by simply $\chi_w(\dot\lambda)$ when working with characters of $\gothH$.  By Proposition \ref{Prop_char}, $\phi_{\lambda}^{\bullet}$ is well defined since the value of $\tilde\chi_{\tilde w}(\dot\lambda)-\chi_w(\dot\lambda)\geq 0$ and integral for all $d_w^{\tilde w}\neq 0$.

\begin{prop}\label{prop_hmorph}The map $\phi_{\lambda}^{\bullet}$ is a ring homomorphism with respect to the product $\odot\bull$.\end{prop}

\begin{proof}Consider the following calculations:

\begin{eqnarray*}\phi_{\lambda}^{\bullet}([\Lambda_u]\odot\bull[\Lambda_v])&=&\phi_{\lambda}^{\bullet}\left(\sum_{w\in W^P} (\tau^{(\chi_w-\chi_u-\chi_v)(\dot\lambda)})\ d_{u,v}^w[\Lambda_w]\right)\\
&=&\sum_{w\in W^P} (\tau^{(\chi_w-\chi_u-\chi_v)(\dot\lambda)})\ d_{u,v}^w\phi_{\lambda}^{\bullet}([\Lambda_w])\\
&=&\sum_{w\in W^P} (\tau^{(\chi_w-\chi_u-\chi_v)(\dot\lambda)})\ d_{u,v}^w\sum_{\tilde w\in\tilW^P}(\tau^{\tilde\chi_{\tilde w}(\dot\lambda)-\chi_w(\dot\lambda)})\ d_w^{\tilde w}[\Lambda_{\tilde w}]\\
&=&\sum_{(w,\tilde w)\in W^P\times\tilW^P}(\tau^{\tilde\chi_{\tilde w}(\dot\lambda)-(\chi_u+\chi_v)(\dot\lambda)})\ d_{u,v}^w d_w^{\tilde w}[\Lambda_{\tilde w}]
\end{eqnarray*}

and

\begin{eqnarray*}\phi_{\lambda}^{\bullet}([\Lambda_u])\odot\bull\phi_{\lambda}^{\bullet}([\Lambda_v])&=&\left(\sum_{\tilde u\in\tilW^P}(\tau^{\tilde\chi_{\tilde u}(\dot\lambda)-\chi_u(\dot\lambda)})\ d_u^{\tilde u}[\Lambda_{\tilde u}]\right)\odot\bull\left(\sum_{\tilde v\in\tilW^P}(\tau^{\tilde\chi_{\tilde v}(\dot\lambda)-\chi_v(\dot\lambda)})\ d_v^{\tilde v}[\Lambda_{\tilde v}]\right)\\
&=& \sum_{(\tilde u,\tilde v)\in(\tilW^P)^2}(\tau^{(\tilde\chi_{\tilde v}+\tilde\chi_{\tilde u})(\dot\lambda)-(\chi_v+\chi_u)(\dot\lambda)})\ d_u^{\tilde u}d_v^{\tilde v}\sum_{\tilde w\in \tilW^P} (\tau^{(\tilde\chi_w-\tilde\chi_u-\tilde\chi_v)(\dot\lambda)})\ d_{\tilde u,\tilde v}^{\tilde w}[\Lambda_{\tilde w}] \\
&=& \sum_{(\tilde u,\tilde v,\tilde w)\in(\tilW^P)^3}(\tau^{\tilde\chi_{\tilde w}(\dot\lambda)-(\chi_u+\chi_v)(\dot\lambda)})\ d_u^{\tilde u}d_v^{\tilde v}d_{\tilde u,\tilde v}^{\tilde w}[\Lambda_{\tilde w}].\end{eqnarray*}

The proposition follows from the fact that $\phi_{\lambda}^*$ is a ring homomorphism.\end{proof}

\subsection*{Definition of $\phi_{\lambda}^{\odot}$ and proof of Theorem \ref{Thm1}}

We define the map $$\phi_{\lambda}^{\odot}:=\phi_{\lambda}^{\bullet}|_{\tau=0}.$$  Clearly this gives a map $$\phi_{\lambda}^{\odot}:H^*(G/\Plam)\rightarrow H^*(\tilG/\tilPl)$$ since the indeterminant vanishes in $(H^*(\tilG/\tilPl)\otimes_{\mathbb{Z}} \mathbb{Z}[\tau],\odot\bull)$.  Define the structure constants $c_w^{\tilde w}$ by expanding with respect to the Schubert basis $$\phi_{\lambda}^{\odot}([\Lambda_w])=\sum_{\tilde w\in\tilW^P}c_w^{\tilde w}[\Lambda_{\tilde w}].$$  By the definition of $\phi_{\lambda}^{\bull}$, we have that $c_w^{\tilde w}=d_w^{\tilde w}$ when $\tilde\chi_{\tilde w}(\dot\lambda)-\chi_w(\dot\lambda)=0$ and $c_w^{\tilde w}=0$ otherwise.
 Moreover, by Lemma \ref{lemma_bkprod} and Proposition \ref{prop_hmorph}, $\phi_{\lambda}^{\odot}$ is a ring homomorphism with respect to the Belkale-Kumar product $\odot_0$ on cohomology.
 \hfill$\Box$

\bigskip

\quad Observe that $c_w^{\tilde w}\neq 0$ if and only if $(w,\tilde w^{\vee})$ is Levi-movable with respect to $\phi_{\lambda}$.  Also if $\phi_{\lambda}$ is the diagonal embedding, we have that $$\phi_{\lambda}^{\odot}([\Lambda_{\tilde u}\times\Lambda_{\tilde v}])=[\Lambda_{\tilde u}]\odot_0[\Lambda_{\tilde v}].$$  The following lemma considers cominuscule flag varieties and is a generalization of \cite[Lemma 19]{BK06}.

\begin{lemma}If $G/P(\lambda)$ is cominuscule, then $\phi_{\lambda}^*$ and $\phi_{\lambda}^{\odot}$ coincide.\end{lemma}

\begin{proof}
Let $w\in W^P$ and $\tilde w\in{\tilde W}^P$.
With notation in Theorem \ref{Thm1}, it is sufficient to prove that if $d_w^{\tilde w}\neq 0$, then $(w,\tilde w)$ is Levi-movable.
Since $d_w^{\tilde w}\neq 0$, there exists $(p,\tilde p)\in P(\lambda)\times\tilde P(\lambda)$ such that the natural map
$$\tilde T\to\frac{T}{pT_w}\oplus\frac{\tilde T}{\tilde p\tilde T_{{\tilde w}^\vee}} $$
is an isomorphism.
 Multiplying $(p,\tilde p)$ by $(\tilde p,{\tilde p}^{-1})$, we may assume that $\tilde p=e$.
 Let us write $p=lu$, with $l\in L(\lambda)$ and $ u$ in the unipotent radical $U(\lambda)$ of $P(\lambda)$.
 Since $G/P(\lambda)$ is cominuscule, $U(\lambda)$ is abelian and $pT_w=lT_w$.
It follows that $(w,\tilde w)$ is Levi-movable.
\end{proof}

\quad The next lemma relates the comorphism $\phi_{\lambda}^{\odot}$ to recent formulas for decomposing structure constants.  The proof is an immediate consequence of \cite[Theorems 1.6 and 1.8]{Ri09}.

\begin{lemma}Multiplicative formulas for decomposing structure constants found in \cite{Re208} and \cite{Ri09} apply to all structure constants associated to the comorphism $\phi_{\lambda}^{\odot}.$\end{lemma}

\section{Examples}

\quad Examples 4.1 and 4.2 require a basic observation on the map $\phi^*_{\lambda}$ restricted to $H^2$.  We remark that this same technique is used by Berenstein and Sjamaar in \cite{BS00}.  Let $\Omega(G)\subseteq\gothH^*$ denote the weight lattice of $G$.
 By \cite{BGG73, De74}, we have that $\Omega(G)$ is isomorphic to $H^2(G/B)$ by mapping $\mu\mapsto c_1(\mathcal{L}_{\mu})$ where $c_1(\mathcal{L}_{\mu})$ is the first Chern class of the line bundle $\mathcal{L}_{\mu}$ with weight $\mu$.
 To any simple root $\alpha_k\in\Delta$, we let $s_k\in W$ denote the corresponding simple reflection and $\pi_k\in \Omega(G)$ denote the corresponding fundamental weight.
 Under the above isomorphism, we have that $\pi_k\mapsto[\Lambda_{s_k^{\vee}}]$.
 Consider the commutative diagram

\begin{equation}\label{lattice eq}\xymatrix{\Omega(G)\ \ar[r]^-{i^*} \ar[d]&\ \Omega(\tilG) \ar[d] \\
H^2(G/B)\ \ar[r]^-{\phi_{\lambda}^2} &\ H^2(\tilG/\tilB)}\end{equation}

where $i^*$ is the induced map from the inclusion $i:\tilde\gothH\hookrightarrow\gothH$.
 It is easy to see that computing $\phi^*_{\lambda}([\Lambda_{s_k^{\vee}}])$ is equivalent to computing $i^*(\pi_k)$.

\subsection{Principal $\SL(\mathbb{C}^2)$ embeddings}
Let $\tilG=\SL(\mathbb{C}^2)$ and $G=\SL(V_n)$ where $V_n$ is the irreducible representation of $\tilG$
associated to the integral weight $n\in\mathbb{Z}_+$ that is of dimension $n+1$.  We remark that the example of $\SL(\mathbb{C}^2)$ embeddings has been studied in \cite[Section 5.3]{BS00}.  Choose the one parameter subgroup $$\lambda(t):=\diag(t,t^{-1})\subseteq\tilG.$$
With respect to the morphism $i:\tilG\rightarrow G$, we have that
$$i\circ\lambda(t)=\diag(t^n,t^{n-2},\ldots,t^{2-n},t^{-n})\subseteq G.$$
Note that, if $n$ is even, $i$ is not injective and one should replace $\SL(\mathbb{C}^2)$ by
${\rm PSL}(\mathbb{C}^2)$.
Here, $\tilPl$ and $\Plam$ are Borel subgroups and hence $\tilG/\tilPl$ is the complex projective line
and $G/\Plam$ is the complete flag variety on $\mathbb{C}^{n+1}$.
 To compute the map

$$\phi_{\lambda}^*:H^*(G/\Plam)\rightarrow H^*(\tilG/\tilPl),$$

we only need to determine $\phi_{\lambda}^*$ restricted to $H^0$ and $H^2$ since $\phi_{\lambda}^*\equiv 0$ on $H^p$ for $p\geq 3$.
 We have that $$H^0(G/\Plam)=\mathbb{Z}[\Lambda_{w_0}]\quad \mbox{and}\quad H^2(G/\Plam)=\bigoplus_{k=1}^n\mathbb{Z}[\Lambda_{s_k^{\vee}}]$$ where $s_1,\ldots s_n$ denote the simple generators of $W$.
 Clearly $$\phi_{\lambda}^*([\Lambda_{w_0}])=[\Lambda_{\tilde w_0}]$$ and using \eqref{lattice eq}, we have that $$\phi^*_{\lambda}([\Lambda_{s_k^{\vee}}])=m_k[\Lambda_{\tilde 1}]$$ where $$m_k:=\sum_{i=1}^k n-2i.$$

Note that for any $k$, the value $m_k=m_{n+1-k}$ and that $m_1=m_{n}=n$.
 We also remark that the sum $\sum_{k=1}^{n}m_k$ is equal to the Dynkin index of the representation $V_n$.

\bigskip

To compute $\phi_{\lambda}^{\odot}$ we determine for which $(w,\tilde w)\in\{(w_0,\tilde w_0),(w_{s_1^{\vee}},1),\ldots,(w_{s_n^{\vee}},1)\}$ we have $$i^*(\chi_w)(\dot\lambda)=\tilde\chi_{\tilde w}(\dot\lambda).$$  Note that $\chi_{w_0}\equiv 0$ and $\tilde\chi_{\tilde w_0}\equiv 0$ since $R^+\cap w_0R^+=\tilR^+\cap \tilde w_0\tilR^+=\emptyset.$  Let $\{x_1,\ldots,x_n\}\in\gothH$ and $\{\tilde x_1\}\in\tilde\gothH$ denote the dual basis to the simple roots $\Delta$ and $\tilde\Delta$ respectively.
 For the pairs  $(w_{s_k^{\vee}},1)$, we observe that $\displaystyle i(\dot\lambda)=2\sum_{i=1}^n x_i\in\gothH$ and $\dot\lambda=2\tilde x_1\in\tilde\gothH$.
 Thus, $$i^*(\chi_{s_k^{\vee}})(\dot\lambda)=2\alpha_{k}\left(\sum_{i=1}^n x_i\right)=2\quad \mbox{and} \quad \tilde\chi_{1}(\dot\lambda)=2\tilde\alpha_1(\tilde x_1)=2$$ since $R^+\cap s_k^{\vee}R^+=\{\alpha_k\}$.
 Hence we have $\phi_{\lambda}^{\odot}=\phi_{\lambda}^*$.
 Note that $G/\Plam$ is not a cominuscule flag variety in this case.

\subsection{Tensor embedding}

Fix an integer $n>0$ and let $\tilG=\SL(\mathbb{C}^n)\times\SL(\mathbb{C}^n)$ and $G=\SL(\mathbb{C}^n\otimes\mathbb{C}^n)$ with the embedding $i:\tilG\hookrightarrow G$ given by the natural action of $\tilG$ on $\mathbb{C}^n\otimes\mathbb{C}^n$.
 Fix integers $k,l< n$ and let $\bar k:=n-k$ and $\bar l:=n-l$.
 Define the one parameter subgroup

$$\lambda(t):=
\diag(\underbrace{t^{\bar k},\ldots,t^{\bar k}}_k,\underbrace{t^{-k},\ldots,t^{-k}}_{\bar k})\times
\diag(\underbrace{t^{\bar l},\ldots,t^{\bar l}}_l,\underbrace{t^{-l},\ldots,t^{-l}}_{\bar l})
\subseteq \tilG.$$

Then
$$i\circ\lambda(t)=\diag(\underbrace{t^{\bar k+\bar l},\ldots,t^{\bar k+\bar l}}_{kl},\underbrace{t^{\bar k-l},\ldots,t^{\bar k-l}}_{\bar kl+k\bar l},\underbrace{t^{-(k+l)},\ldots,t^{-(k+l)}}_{\bar k\bar l})\subseteq G.$$

Here we have that $\tilG/\tilPl$ is the product of Grassmannians $\Gr(k,\mathbb{C}^n)\times\Gr(l,\mathbb{C}^n)$ and $G/\Plam$ is the two-step flag variety $\Fl(kl, n^2-\bar k\bar l;\mathbb{C}^n\otimes\mathbb{C}^n).$  In general the map $\phi^*_{\lambda}$ is quite difficult to explicitly compute.
 We will compute  $\phi^*_{\lambda}$ restricted to $H^2$.
 With respect to the Schubert basis, we have that

$$H^2(G/\Plam)=\mathbb{Z}[\Lambda_{s_{kl}^{\vee}}]\oplus\mathbb{Z}[\Lambda_{s_{n^2-\bar k\bar l}^{\vee}}]\simeq \mathbb{Z}^2$$

where $s_{kl}, s_{n^2-\bar k\bar l}$ denote the simple reflections in $W^P$ and

$$H^2(\tilG/\tilPl)=\mathbb{Z}[\Lambda_{\tilde w_1}]\oplus\mathbb{Z}[\Lambda_{\tilde w_2}]\simeq \mathbb{Z}^2$$

where $$\Lambda_{\tilde w_1}:=\Lambda_{\tilde s_k^{\vee}}\times\Gr(l,\mathbb{C}^n)\quad\mbox{and}\quad \Lambda_{\tilde w_2}:=\Gr(k,\mathbb{C}^n)\times\Lambda_{\tilde s_l^{\vee}}.$$  Using \eqref{lattice eq}, we find that

$$\phi_{\lambda}^*([\Lambda_{s_{kl}^{\vee}}])=l[\Lambda_{\tilde w_1^{\vee}}]+k[\Lambda_{\tilde w_2^{\vee}}]$$

and $$\phi_{\lambda}^*([\Lambda_{s_{n^2-\bar k\bar l}^{\vee}}])=
\bar l[\Lambda_{\tilde w_1^{\vee}}]+\bar k[\Lambda_{\tilde w_2^{\vee}}].$$

Let $\{x_1,\ldots, x_{n^2-1}\}$ and $\{\tilde x_1,\ldots\tilde x_{n-1},\tilde x_1',\ldots,\tilde x_{n-1}'\}$ denote the dual basis to $\Delta$ and $\tilde\Delta$ respectively.
Writing $\dot\lambda$ in terms of this basis gives
$$\dot\lambda=n(\tilde x_k+\tilde x_l')\quad \mbox{and}\quad i(\dot\lambda)=
n(x_{kl}+x_{n^2-\bar k\bar l}).$$
Computing the characters $\chi$ gives,
$$\tilde\chi_{w_1}=\tilde\alpha_k,\quad \tilde\chi_{w_2}=\tilde\alpha_l'$$
and
$$\chi_{s_{kl}}=\alpha_{kl},\quad \chi_{s_{n^2-\bar k\bar l}}=\alpha_{n^2-\bar k\bar l}.$$

Hence
$$\tilde\chi_{w_1}(\dot\lambda)=\tilde\chi_{w_2}(\dot\lambda)=
i^*(\chi_{s_{kl}})i(\dot\lambda)=i^*(\chi_{s_{n^2-\bar k\bar l}})(\dot\lambda)=n.$$
Thus $\phi_{\lambda}^{\odot}=\phi_{\lambda}^*$ restricted to $H^2$.

\subsection{Odd orthogonal embedding} Fix a positive integer $m$ and let $n=2m+1.$  Let $\tilG=\SO(\mathbb{C}^n)$ denote the special orthogonal group on $\mathbb{C}^n$ with respect to the quadratic form
$$
Q(\sum t_ie_i):=t_{m+1}^2+\sum_{i=1}^m t_it_{2m+2-i}
$$
where $\{e_1,\ldots,e_n\}$ is the standard basis of $\mathbb{C}^n$.
Let $G=\SL(\mathbb{C}^n)$ and let $i:\tilG\hookrightarrow G$ be the natural embedding of groups.
Fix an integer $k\leq m$ and define the one parameter subgroup
$$\lambda(t):=\diag(\underbrace{t,\ldots,t}_k,\underbrace{0,\ldots,0}_{n-2k},\underbrace{t^{-1},\ldots,t^{-1}}_{k})
\subseteq \tilG.$$
It is easy to see that $i\circ\lambda(t)\subseteq G$ has the same presentation as $\lambda(t)$ above.  Here we have that $\tilG/\tilPl$ is the orthogonal Grassmannian $\OG(k,\mathbb{C}^n)$ of isotropic $k$-planes in $\mathbb{C}^n$ with respect to $Q$ and $G/\Plam$ is equal to the two step flag variety $\Fl(k,n-k;\mathbb{C}^n)$.  The embedding $\phi_{\lambda}$ is given by $\phi_{\lambda}(V)=(V,V^{\perp})$ where $V$ is an isotropic $k$-plane in $\mathbb{C}^n$ and $V^{\perp}$ denotes the orthogonal complement of $V$ in $\mathbb{C}^n$.  While $\phi_{\lambda}^*$ is very difficult to determine in general, we can compute $\phi_{\lambda}^*([\Lambda_w])$ for certain $w\in W^P$.  Consider the diagram

\begin{equation}\label{ortho diagram}\xymatrix{\OG(k,\mathbb{C}^n)\ \ar[r]^-{\phi_{\lambda}} \ar[dr]_-{\psi_1}&\ \Fl(k,n-k,\mathbb{C}^n) \ar[d]^-{\psi_2} \\
  &\ \ \Gr(k,\mathbb{C}^n)}\end{equation}

where $\psi_1$ is the natural inclusion of $\OG(k,\mathbb{C}^n)$ in $\Gr(k,\mathbb{C}^n)$ and $\psi_2$ is the natural projection of $\Fl(k,n-k,\mathbb{C}^n)$ onto $\Gr(k,\mathbb{C}^n)$.  In \cite{Co09}, Coskun gives a branching algorithm which determines the map $\psi_1^*$ on cohomology with respect to the Schubert basis.  By the commutivity of diagram (\ref{ortho diagram}), we can compute $\phi^*_{\lambda}([\Lambda_w])$ for any Schubert class that can be written as $[\Lambda_w]=\psi_2^*([\Lambda_{w'}])$ for some Schubert class $[\Lambda_{w'}]\in H^*(\Gr(k,n)).$

\smallskip

\quad For the following example, we adopt the notation found in \cite[Chapter 3]{BiLa00}.  Let $n=9$ and $k=3$.  In this case, we can identify the Weyl group $W$ with the symmetric group $S_9$ and $\tilW$ can be identified with the subgroup of $S_9$ given by
$$\tilW=\{(a_1\cdots a_9)\in S_9\ |\ a_i+a_{10-i}=10\}.$$  Let $w=(468579123)\in W^P$.  Then $[\Lambda_w]=\psi_2^*([\Lambda_{w'}])$ where $w'=(468123579)$.  Hence $$\phi^*_{\lambda}([\Lambda_w])=\phi^*_{\lambda}\circ\psi_2^*([\Lambda_{w'}])=\psi_1^*([\Lambda_{w'}]).$$  By \cite[Example 4.4]{Co09}, we have that

$$\psi_1^*([\Lambda_{w'}])=4[\Lambda_{\tilde w_1}]+2[\Lambda_{\tilde w_2}]+2[\Lambda_{\tilde w_3}]$$ where $$\tilde w_1=(348159267),\quad \tilde w_2=(168357249),\quad \tilde w_3=(267159348).$$

Let $(x_1,\ldots x_8)$ and $(\tilde x_1,\tilde x_2,\tilde x_3)$ denote the dual basis to $\Delta$ and $\tilde\Delta$ respectively.  Writing $\dot\lambda$ in terms of this basis gives
$$\dot\lambda=\tilde x_3\quad \mbox{and}\quad i(\dot\lambda)= x_3+x_6.$$  By \cite[Theorem 1(i)]{Ri08} and the definition of $\chi_w$ we have that $$i^*(\chi_w)(\dot\lambda)=\chi_w(x_3)+\chi_w(x_6)=6.$$ and by an odd orthogonal analogue of \cite[Lemma 50]{BK07} we have that  $$\tilde\chi_{\tilde w_1}(\dot\lambda)=\tilde\chi_{\tilde w_2}(\dot\lambda)=\tilde\chi_{\tilde w_3}(\dot\lambda)=9.$$

Hence $\phi_{\lambda}^{\odot}([\Lambda_w])=0$.  Observe that in this case $\phi_{\lambda}^{\odot}\neq \phi_{\lambda}^*.$

\section{Application to Eigencones}\label{section eigen}

\quad In this section, we make the assumption that no ideal of $\tgothG$ is an ideal of $\gothG$.

\bigskip

\quad Let $\Chi(H)$ denote the group of characters of $H$ and set $\Chi(H)_\QQ:=\Chi(H)\otimes_\ZZ\QQ$.  If $\nu\in \Chi(H)$ is dominant, we will denote by $V_\nu$ the irreducible representation of highest weight $\nu$. We will use similar notation for $\tG$.

\smallskip

\quad We denote by $\lr(\tG,G)$ 
the cone of the pairs $(\tnu,\nu)\in \Chi(\tH)_\QQ\times\Chi(H)_\QQ$
such that $n\tnu$ and $n\nu$ are dominant 
weights and $V_{n\tnu}\otimes V_{n\nu}$ contains nonzero $\tG$-invariant vectors for some positive integer $n$.
The set $\lr(\tG,G)$ is a closed convex rational polyhedral cone contained in the dominant chamber
$\Chi(\tH)_\QQ^+\times\Chi(H)_\QQ^+$.
Moreover, by \cite[Proposition~]{Re07}, our assumption implies that $\lr^\circ(\tG,G)$ has
non empty interior.  The aim of this section is to describe 
$\lr^\circ(\tG,G)$ as a part of $\Chi(\tH)_\QQ^+\times\Chi(H)_\QQ^+$ by a minimal list of inequalities.
We first introduce some notation.

\smallskip

\quad Let $\Wt_\tH(\gothG/\tgothG)$ be the set of the nontrivial weights for the $\tH$-action on $\gothG$.
Let $\Chi(\tH)\otimes \QQ$ denote the rational vector space  generated the characters of $\tH$.
We consider the set of hyperplanes ${\mathcal H}$ of $\Chi(\tH)\otimes \QQ$ spanned by elements of $\Wt_\tH(\gothG/\tgothG)$.
For each such hyperplane $h\in{\mathcal H}$ there exists exactly two opposite indivisible one parameter subgroups
$\pm\lambda_{h}$ which are orthogonal (for the parring $\langle\cdot,\cdot\rangle$) to $h$.
These one parameter subgroups of $\tH$ form a set stable by the action of $\tW$. Let $\{\lambda_1,\,\cdots,\lambda_n\}$ be the set of dominant one parameter subgroups obtained from the hyperplanes in ${\mathcal H}$.

\begin{thm}
\label{th:param2}
  We assume that no ideal of $\tgothG$ is an ideal of $\gothG$.

\smallskip

A point $(\tnu,\nu)\in \Chi(\tH)_\QQ^+\times\Chi(H)_\QQ^+$ belongs to  $\lr(\tG,G)$ if and only if
for all $i=1,\cdots, n$
and for all pair of Schubert classes $([\tLambda_\tw],\,[\Lambda_w])$ of $\tG/\tP(\lambda_i)$ and
$G/P(\lambda_i)$ associated to  $(\tw,w)\in \tW^{\tP(\lambda_i)}\times W^{P(\lambda_i)}$ such that

\begin{eqnarray}
\phi_{\lambda_i}^\odot([\Lambda_w])\odot_0[\tLambda_\tw]=[{\tLambda_e}]\in {\rm H}^*(\tG/\tP(\lambda_i),\ZZ),
\end{eqnarray}
we have
\begin{eqnarray}
  \label{eq:ineg}
\langle \tw\lambda_i, \tnu \rangle+\langle w\lambda_i,\nu\rangle  \geq 0.
\end{eqnarray}
Moreover, one can omit no inequalities in the above list.
\end{thm}

\begin{proof}
  Let $\rho$ and $\trho$ denote the half sum of all positive roots of  $\gothG$ and $\tgothG$ respectively.  By \cite[Theorems~A and B]{Re07}, we only have to prove that
 $\phi_{\lambda_i}^\odot([\Lambda_w])\odot_0[\tLambda_\tw]
=[{\rm pt}]$ if and only if
 $\phi_{\lambda_i}^*([\Lambda_w]).[\tLambda_\tw]
=[{\rm pt}]$ and ``$\langle \tw\lambda_i,\trho\rangle+\langle w\lambda_i,\rho\rangle=
\langle \lambda_i,\trho\rangle+\langle\lambda_i,2\rho^{\lambda_i}-\rho\rangle $ ''
(with notation of \cite{Re07}).
This follows immediately from Proposition~\ref{Prop_char}.
\end{proof}

\begin{rem}
In \cite{Re08}, the first author gives a bijective parametrization of the faces of $\lr(\tG,G)$ which intersect
the interior of the dominant chamber. The morphism $\phi_\lambda^\odot$ can also be used to simplify
the statements of \cite{Re08}.
For example, with notation of \cite[Paragraph~7.2.3]{Re08}, the conditions
\begin{enumerate}
\item $\phi_\lambda^*([\overline{BwP(\lambda)/P(\lambda)}])\cdot[\overline{\tilde B\tP(\lambda)/\tP(\lambda)}]
=[{\rm pt}]\in{\rm H}^*(\tG/\tP(\lambda),\ZZ)$;
\item $(\theta^{P(\lambda)}_w)_{|\tS}=(\theta^{\tP(\lambda)}-2(\trho-\trho^\tS))_{|\tS}$.
\end{enumerate}
are equivalent to
$$
\phi_\lambda^\odot([\overline{BwP(\lambda)/P(\lambda)}])\odot_0[\overline{\tilde B\tP(\lambda)/\tP(\lambda)}]
=[{\rm pt}]\in{\rm H}^*(\tG/\tP(\lambda),\ZZ).
$$
\end{rem}

\bibliographystyle{abbrv}
\bibliography{references}

\end{document}